\documentclass[10pt]{amsart}
\usepackage{amsmath,amssymb,amsthm}
\newtheorem{thm}{Theorem}[section]
\newtheorem{lem}[thm]{Lemma}
\newtheorem{prop}[thm]{Proposition}
\newtheorem{cor}[thm]{Corollary}
\numberwithin{equation}{section}

\theoremstyle{example}

\theoremstyle{definition}

\theoremstyle{notation}
\newtheorem{notation}[thm]{Notation}
\theoremstyle{remark}
\newtheorem{remark}[thm]{Remark}

\newcommand{\C}{{\mathbb C}}
\newcommand{\N}{{\mathbb N}}
\newcommand{\Z}{{\mathbb Z}}
\newcommand{\R}{{\mathbb R}}

\def\11{{\rm 1~\hspace{-1.4ex}l} }

\title
[Invariant measures for BO ]
{Invariant measures and long time behaviour for the  Benjamin-Ono equation}
\author[Nikolay Tzvetkov]{Nikolay~Tzvetkov}
\author[Nicola Visciglia]{Nicola Visciglia}
\address{D\'epartement de Math\'ematiques, Universit\'e de Cergy-Pontoise, 2, 
avenue Adolphe Chauvin, 95302 Cergy-Pontoise  
Cedex, France and Institut Universitaire de France}\email{nikolay.tzvetkov@u-cergy.fr}
\address{Universit\`a Degli Studi di Pisa, Dipartimento di Matematica,
Largo Bruno Pontecorvo 5 I - 56127 Pisa. Italy}\email{ viscigli@dm.unipi.it}
\begin{document}
\maketitle

\begin{abstract}
We study the Benjamin-Ono equation, posed on the torus.
We prove that an infinite sequence of weighted gaussian measures, constructed in our previous work, are invariant by the flow of the equation. 
These measures are supported by Sobolev spaces of increasing regularity. As a by product we deduce informations on the long time behaviour  
of regular solutions.
To our knowledge this is the first result which gives an evidence about recurrence properties of the Benjamin-Ono equation flow. 
\end{abstract}


\section{Introduction}
This paper is a sequel of our previous works \cite{tz, TV}. It deals with the long time behaviour of the solutions of the Benjamin-Ono equation, posed on the torus.
The Benjamin-Ono equation is a fundamental dispersive equation modeling the propagation of long small amplitude internal waves. It is less dispersive than
the famous KdV equation (which models surface waves).  
Let us recall that the Cauchy problem analysis of this equation turned out to be quite interesting (see \cite{MST, BP,IK,M,T, MP}).
To our best knowledge, the long time behaviour in the periodic case for large data is a widely open problem. 
The main goal of this work is to make a progress on this question by constructing 
invariant measures. Therefore, thanks to Poincar\'e's theorem, we show  an evidence about recurrence properties of the Benjamin-Ono equation flow.
We point out that the measures are supported by Sobolev spaces of increasing regularity and consequently our result is of importance for the dynamics of regular solutions as well. We also note that for the KdV equation  more precise evidences of the recurrence of the flow are known (see e.g.  \cite{MT,KP,B0}).
Our approach uses heavily a probabilistic view point, both on the measure construction and on the measure invariance proof.
In particular, the arguments used in the present paper are less dependent on the properties of individual solutions
compared to previous works on invariant measures for dispersive equations (see e.g. \cite{zh,B,B2, BT1, BT2, BTT, tz_fourier,NORS, OH}). 
This roughly explains why the approach works even in such a weak dispersion situation. 
We hope that this aspect of our analysis may be useful in other contexts.

Consider thus the Benjamin-Ono equation
\begin{equation}\label{bo}
\partial_t u + H\partial_x^2 u + u\partial_x u=0
\end{equation}
where $H$ denotes the Hilbert transform, posed on the torus $\R|(2\pi\Z)$.
The Sobolev spaces are natural phase spaces for \eqref{bo}.
We have that the mean value $\int u$ is conserved under the flow of \eqref{bo}.
Hence it is not restrictive to study \eqref{bo} for initial data of zero mean value (no zero Fourier coefficient).
Indeed the general case can be reduced to the zero mean value case by considering the problem solved by $u(t)-\int u(0)$, which is \eqref{bo} plus a harmless transport term. Therefore we will consider zero mean value solutions of \eqref{bo} and we shall denote by $H^s$ the Sobolev space of zero mean value functions equipped with the usual norm. Thanks to the work of Molinet \cite{M} (see \cite{ABFS} for earlier results) the problem \eqref{bo} is globally well-posed in the Sobolev spaces $H^s$, $s\geq 0$. We note by $\Phi_t$, $t\in\R$ the flow established in \cite{M} and 
for every subset $A\subset H^s$ (with $s\geq 0$ fixed) and for every $t\in \R$ we
define the set $\Phi_t(A)$ as follows:
\begin{equation}\label{eq:PhiBO}\Phi_t(A)=\{u(t,.)\in H^s|  \hbox{ where }
u(t,.) \hbox{ solves } \eqref{bo} \hbox{ with } u(0,.)\in A\}.
\end{equation}
We now recall some notations from our previous paper \cite{TV}.
Smooth solutions to \eqref{bo} satisfy infinitely many conservation laws (see e.g. \cite{Mats, ABFS}).
More precisely for $k\geq 0$ an integer, there is a conservation law of \eqref{bo} of the form 
\begin{equation}\label{strucRmezzi}
E_{k/2}(u)= \|u\|_{\dot{H}^{k/2}}^2 + R_{k/2} (u) 
\end{equation}
where $\dot{H}^s$ denotes the homogeneous Sobolev norm on periodic functions and all the terms that appear in
$R_{k/2}$  are homogeneous in $u$ of order  larger or equal to three.
In the spirit of the works \cite{zh,B, LRS}, we shall define invariant measures for \eqref{bo} 
by re-normalizing the formal measure $\exp(-E_{k/2}(u))du$.
Denote by $\mu_{k/2}$ the gaussian measure induced 
by the random Fourier series 
\begin{equation}\label{randomized}
\varphi_{k/2}(x, \omega)=\sum_{n\in \Z \setminus \{ 0\}} 
\frac{\varphi_n(\omega)}{|n|^{k/2}} e^{{\bf i}nx}.
\end{equation}
In \eqref{randomized}, $(\varphi_{n}(\omega))$ is a sequence of centered complex gaussian variables 
defined on a probability space $(\Omega, {\mathcal A}, p)$ such that $\varphi_{n}=\overline{\varphi_{-n}}$
(since the solutions of \eqref{bo} should be real valued) and  
$(\varphi_{n}(\omega))_{n>0}$ are independent.
More precisely, we have that for a suitable constant $c$, $\varphi_{n}(\omega)=c(h_n(\omega)+{\bf i}l_n(\omega))$,
where $h_n,l_n\in {\mathcal N}(0,1)$ are standard real gaussians.
We have that $\mu_{k/2}(H^s)=1$ for every $s<(k-1)/2$ while $\mu_{k/2}(H^{(k-1)/2})=0$, i.e. for large $k$ the support of $\mu_{k/2}$ contains quite regular functions. 
For any $N\geq 1$, $k\geq 2$ and $R>0$ we introduce the function
\begin{equation}\label{Femme+1}
F_{k/2,N, R}(u)
=\Big(\prod_{j=0}^{k-2} \chi_R (E_{j/2}(\pi_N u)) \Big)
\chi_R (E_{(k-1)/2}(\pi_N u)-\alpha_N) 
e^{-R_{k/2}(\pi_N u)}
\end{equation}
where
$\alpha_N=\sum_{n=1}^N \frac {c}{n}$ for a suitable constant $c$, $\pi_N$ denotes the Dirichlet projector 
on Fourier modes $n$ such that $|n|\leq N$, $\chi_{R}$ is a cut-off function defined as $\chi_R(x)=\chi (x/R)$
with $\chi:\R \rightarrow \R$ a smooth,
compactly supported function such that $\chi(x)=1$ for every $|x|<1$.
Next we state the main result proved in \cite{TV}.
\begin{thm}\label{mainold}
For every $k\in \N$ with $k\geq 2$ there exists a $\mu_{k/2}$ measurable function 
$F_{k/2,R }(u)$ such that $F_{k/2,N, R}(u)$ converges to $F_{k/2,R }(u)$ in $L^q(d\mu_{k/2})$ 
for every $1\leq q<\infty$.
In particular $F_{k/2,R }(u)\in L^q(d\mu_{k/2})$. 
Moreover, if we set $d\rho_{k/2,R}\equiv F_{k/2,R }(u)d\mu_{k/2}$ then we have
$$
\bigcup_{R>0}{\rm supp}(\rho_{k/2,R})={\rm supp}(\mu_{k/2}).
$$
\end{thm}
Our main contribution in this paper is the proof of the invariance
of the measures $\rho_{k/2,R}$
constructed in the previous theorem, provided that $k\geq 6$ is an even integer (a fact conjectured in \cite{TV}).
\begin{thm}\label{lem4}
For every even integer $k\geq 6$ and for every $R>0$ the measures $\rho_{k/2,R}$ 
are invariant under the flow associated with \eqref{bo}. More precisely
for every Borel set $A\subset H^\sigma$
with $2\leq \sigma<(k-1)/2$, and for every $t_0\in \R$ we have
$$\int_A  F_{k/2,R }(u)d\mu_{k/2} = \int_{\Phi_{t_0}(A)}  F_{k/2,R }(u)d\mu_{k/2}.$$ 
\end{thm}
Let us explain the main steps in the proof of Theorem~\ref{lem4}.
Once the existence of the measures  $\rho_{k/2,R}$ is established via a delicate renormalization procedure (see the main result in \cite{TV}), the basic
difficulty in order to prove their invariance comes from the fact that the energies $E_{k/2}$, that are conserved for the equation \eqref{bo}, 
are no longer conserved for the approximated problems (see \eqref{BON} below) as long as $k\geq 2$. 
However they are formally conserved in a suitable asymptotic sense which in the
Benjamin-Ono case is very weak.
Such an asymptotic conservation property is quite delicate (if possible) to be established for individual solutions on the support of $\rho_{k/2,R}$. 
Here we prove such an asymptotic conservation property only in an averaged sense and thus  
the main point in the proof of Theorem~\ref{lem4} is to reduce the analysis at time $t=0$. 
This is possible thanks to a key property at $t=0$ first introduced in our previous work \cite{TV} 
which enables one to invert the limit as the dimension goes to infinite 
with the limit occurring in the time derivatives of the energies $E_{k/2}$ along the truncated flows.
We also underline that the deterministic estimates used in this paper are rather classical
since we are mainly focused on high order conservation laws. 
In the proof of Theorem~\ref{lem4} it is of importance that we use the approximation flows, first introduced by Burq-Thomann and the first author in \cite{BTT}.
More precisely, in contrast with previous works as \cite{zh, B,tz_fourier}, we see the truncated flows and the truncated measures on a fixed infinite dimensional space (a suitable Sobolev space). 
This makes the limit properties of the truncated measures more flexible. 
The additional input in the analysis, compared to \cite{zh,B,tz_fourier}, 
is the invariance of complex gaussians under rotations, a fact applied to the evolution of the high frequencies modes of the truncated problems.

In \cite{zh} and \cite{NORS}, a similar difficulty of lack of conservation of the approximated problems occurs. 
We point out that in this paper, we are forced to solve this problem quite differently compared to \cite{zh,NORS}.
In \cite{zh, NORS} this problem is solved by establishing
energy growth estimates for individual solutions on the support of the measure. 
Such a deterministic approach meets serious difficulties in the context of the 
Benjamin-Ono equation and after spending a considerable amount of time in trying to follow this approach, we have been obliged to exploit the fact that asymptotic conservation property occur only in a weaker averaged sense. Let us emphasize that an analogue of \cite[Theorem 4.2]{NORS} for the Benjamin-Ono case is not used in 
our work and it is not clear whether such a property holds for individual solutions.   

As already mentioned, thanks to the Poincar\'e recurrence theorem (see e.g. \cite{zh}), we have the following corollary of Theorem~\ref{lem4}.
\begin{cor}\label{corollary}
Let $k\geq 6$ be an even integer and $0\leq \sigma<(k-1)/2$.
Then the solutions of  the Benjamin-Ono equation \eqref{bo} are recurrent in the following sense: for $\mu_{k/2}$ almost every $u_0\in H^\sigma$  
there exists a sequence of times $(t_n)_{n\geq 0}$ going to infinity such that 
$$
\lim_{n\rightarrow\infty}\|\Phi_{t_n}(u_0)-u_0\|_{H^\sigma}=0.
$$
\end{cor}
Let us notice that very few results of this type are known for 
nonlinear hamiltonian PDE's with {\it large and  smooth} data . 
On one hand the KAM theory gives informations on the recurrence of the flow for small data. On the other hand the 
the invariant measures results provide recurrence properties 
at low regularity.
In this discussion a notable exception is the KdV equation for which it is known (see \cite{B2}, \cite{MT}) 
that the flow is almost periodic for $H^s$, $s\geq 0$ data which implies that for the KdV equation, 
the recurrence property displayed by Corollary~\ref{corollary} holds for {\it every} $H^s$, $s\geq 0$ function as initial datum. 
Establishing a similar property for the Benjamin-Ono equation is a challenging problem.

We believe that the result of Theorem~\ref{lem4} is true for every $k$ (even or odd). 
Here, we decided to restrict our attention only to the case of even $k\geq 6$ since 
{\it it already contains the phenomenon we would like to describe} and it avoids us to enter in
technicalities which would deviate us of the main message of this paper. 
Let us briefly explain what we think remains to be done in order to get the invariance of all measures $\rho_{k/2,R}$.
In the case of odd $k\geq 7$ one should rework the second main result of \cite{TV}. This would require an additional orthogonality 
argument compared with \cite{TV}. In the case of small $k$ a more sophisticated deterministic analysis, 
related with the low regularity well-posedness theory of the Benjamin-Ono equation, 
should be involved. \footnote{After this work was completed, a paper by Deng \cite{Deng} solving the problem for $k=1$ appeared.} 

The rest of the manuscript is devoted to the proof of Theorem~\ref{lem4}.
Next we fix some notations.
\begin{notation}
For every $N$ we denote by $\pi_N$ the projector on the first $n$ Fourier modes with
$|n|\leq N$ and $\pi_{>N}= 1-\pi_N$.\\
For every $\rho \in \R, r\geq 0$ we set $$B^\rho(r)=\{u\in H^\rho| \|u\|_{H^\rho}<r\}.$$
We denote by $\Phi_{t}$ the flow associated with the Benjamin-Ono equation. The corresponding truncated flow
$\Phi_{t}^N$ will be defined along section \ref{detsec}. \\
We denote by ${\mathcal B} (H^s)$ the $\sigma$-algebra of Borel subsets in $H^s$.\\
The randomized vector $\varphi_{k/2}(\omega,x)$ is defined in \eqref{randomized}
with $\omega$ delonging to the probability space $(\Omega, {\mathcal A}, p)$.
We denote by $L^q_\omega$ the associated Lebesgue spaces $L^q(\Omega, {\mathcal A},p)$.

\end{notation}
{\bf Acknowledgement.}
{\em Nikolay Tzvetkov is supported by the ERC project Dispeq,
Nicola Visciglia is supported by the FIRB project Dinamiche Dispersive:
Analisi di Fourier e Metodi Variazionali. The authors are grateful to the referee for the valuable remarks which
helped to improve the exposition.}
\section{On the structure of conservation laws}\label{matconla}
The main result of this section is Proposition~\ref{prop:cubstruc}. 
First we recall some notations introduced in \cite{TV}
to describe the structure of the conservation laws satisfied by solutions to \eqref{bo}
(for more details see Section~2 in \cite{TV}).\\\\
Given any function $u(x)\in C^\infty(S^1)$, we define
\begin{eqnarray*}
{\mathcal P}_1(u) & = & \{\partial_x^{\alpha_1} u, 
H\partial_x^{\alpha_1} u|\alpha_1\in \N\},
\\
{\mathcal P}_2(u) & = & \{\partial_x^{\alpha_1} 
u\partial_x^{\alpha_2} u, 
(H \partial_x^{\alpha_1} u)\partial_x^{\alpha_2} u,
(H \partial_x^{\alpha_1} u)(H \partial_x^{\alpha_2} u)|\alpha_1,\alpha_2\in \N\}
\end{eqnarray*}
and in general by induction
\begin{multline*}
{\mathcal P}_n(u)=\Big \{\prod_{l=1}^k H^{i_l}p_{j_l}(u)|
i_1,...,i_k\in \{0,1\}, 
\\
\sum_{l=1}^k j_l=n, k\in \{2,...,n\}
\hbox{ and } p_{j_l}(u)\in {\mathcal P}_{j_l}(u)\Big \}
\end{multline*}
where $H$ is the Hilbert transform.
\begin{remark}
Roughly speaking an element in 
${\mathcal P}_n(u)$ involves the product of 
$n$ derivatives $\partial_x^{\alpha_1} u, .., 
\partial^{\alpha_n}_x u$ 
in combination with the Hilbert transform $H$ 
(that can appear essentially in an arbitrary way in front of the factors and eventually in front of a group of factors).
\end{remark}
Notice that for every $n$ the simplest element belonging 
to $\mathcal P_n(u)$ has the following structure:
\begin{equation}\label{ptilde}
\prod_{i=1}^n \partial_x^{\alpha_i} u, \alpha_i\in \N.
\end{equation}
In particular we can define the map
$${\mathcal P}_n(u)\ni p_n(u)\rightarrow 
\tilde p_n(u)\in {\mathcal P_n}(u)$$
that associates to every $p_n(u) \in {\mathcal P}_n(u)$ the unique element
$\tilde p_n(u)\in {\mathcal P}_n(u)$ having the structure given in \eqref{ptilde}
where $\partial_x^{\alpha_1}u, \partial_x^{\alpha_2}u,
..., \partial_x^{\alpha_n}u$
are the derivatives involved in the expression of $p_n(u)$
(equivalently $\tilde p_n(u)$ is obtained from $p_n(u)$ by 
erasing all the Hilbert transforms $H$ that appear in $p_n(u)$).\\
Next, we associate to every $p_n(u)\in {\mathcal P_n}(u)$ 
two integers as follows:\\
$$\hbox{ if }\tilde p_n(u)=\prod_{i=1}^n \partial_x^{\alpha_i} u
\hbox{ then }$$ 
\begin{equation}\label{norm}|p_n(u)|:=\sup_{i=1,..,n} \alpha_i
\end{equation}
and
\begin{equation}\label{doublenorm}\|p_n(u)\|:=\sum_{i=1}^n \alpha_i.\end{equation}
We are ready to describe the structure of the 
conservation laws satisfied by the
Benjamin-Ono equation.
Given any even $k\in \N$, i.e. $k=2n$, the energy $E_{k/2}$ 
has the following structure:
\begin{multline}\label{even}E_{k/2}(u)
=  \|u\|_{\dot{H}^{n}}^2 + 
\sum_{\substack{p(u)\in {\mathcal P}_3(u) s.t. \\ 
\tilde p(u)=u\partial_x^{n-1} u \partial_x^{n}u}}
c_{2n}(p) \int p(u)dx
\\
+
\sum_{\substack{p(u)\in {\mathcal P}_{j}(u) s.t. j=3,..., 2n+2\\ 
\|p(u)\|= 2n-j+2\\ |p(u)|\leq n-1}} c_{2n}(p) \int p(u)dx 
\end{multline}
where $c_{2n}(p)\in \R$ are suitable real numbers.
Observe that the above representation is not unique.
For example 
$
\int u\partial_{x}^{n-1}u\partial_{x}^{n}u
$
which is a priori in the second term in the right hand-side of \eqref{even} can be written, after an integration by parts as
$
-\frac{1}{2}\int \partial_x u(\partial_{x}^{n-1}u)^2
$
which transfers it to the third term in the right hand-side of \eqref{even}.\\
For the sake of completeness and since we shall need it in the sequel we recall
that for $k\in \N$ odd, i.e. $k=2n+1$, 
the energy $E_{k/2}$ has the following structure:
\begin{equation}\label{odd}E_{k/2}(u)
=  \|u\|_{\dot{H}^{n+1/2}}^2 
+ 
\sum_{\substack{p(u)\in {\mathcal P}_j(u) s.t.j=3,...,2n+3\\ 
\|p(u)\|=2n-j+3\\|p(u)|\leq n}} c_{2n+1}(p) \int p(u)dx
\end{equation}
where $c_{2n+1}(p)\in \R$ are suitable real numbers.\\

The main result of this section is the following proposition concerning the structure of $E_{k/2}$ with $k$ even.
\begin{prop}\label{prop:cubstruc}
Let $k=2(m+1)$. Then one may assume that the only term of the second term in the right hand-side of \eqref{even} is given by
$$
 c \int u (H \partial_x^m u) \partial_x^{m+1}u dx 
$$
for a suitable constant $c$.
\end{prop}
\begin{remark}
A similar statement holds for $k$ odd. We do not include it here, since as we already mentioned, in order to avoid some additional technicalities, we decided to restrict
our attention to $k$ even (and large). 
\end{remark}
\begin{proof}
We shall follow the Matsuno book \cite{Mats} where the structure of conservation laws satisfied by solutions of
\begin{equation}\label{bo4}
\partial_t u + H\partial_x^2 u + 4 u\partial_x u=0
\end{equation}
is studied.
Notice that $u$ solves \eqref{bo4} iff $\frac{1}{4} u$ solves \eqref{bo}. As a consequence one can check that it is sufficient
to prove the proposition by assuming that the energy $E_{m+1}$ 
is the one preserved by solutions to \eqref{bo4} (and not by solutions to \eqref{bo}, which is the true equation we are interested in).
In fact, the structure of the conservation laws respectively associated with \eqref{bo}
and \eqref{bo4}, are strictly related modulo some multiplicative factors
which are suitable powers of $1/4$. 

Following Matsuno we have that the conservation laws $E_{k/2}$ (satisfied by solutions to \eqref{bo4})
are obtained as follows.
First, given any function $u$, we introduce the power series
\begin{equation}\label{red}
w(u)=\sum_{n=1}^\infty w_n(u) \epsilon^n
\end{equation}
where $w_n(u)$ are densities that satisfy
\begin{equation}\label{idmatsuno}-\epsilon {\bf i} P_- w_x + (1- e^{-w})=\epsilon u
\end{equation}
and $P_{-}=\frac{1}{2}(1-{\bf i}H)$ is the projector on negative frequencies.
Then the quantities $\int w_n(u) dx$ are preserved along the evolution of the Benjamin-Ono equation
\eqref{bo4}.
Notice that in this language the conservation laws are parametrized by the natural numbers $n\in \N$
(and not by the rationals $k/2$ with $k\in \N$). More precisely 
the conservation law $E_{k/2}$ (for every $k\in N$) corresponds to $\int w_{k+2}(u) dx$.
\\
Notice that the content of the proposition concerns the expression
$\int w_{2m+4}(u) dx$.
By developing $e^{-w}$ then by \eqref{idmatsuno} we get:
\begin{equation}\label{idmatsuno2}
-\epsilon {\bf i} P_- w_x + (w-\frac{w^2}{2!}+\frac{w^3}{3!}-...)=\epsilon u.
\end{equation}
By inserting \eqref{red} in \eqref{idmatsuno2} and computing the terms appearing in front of the corresponding powers of $\epsilon$, we obtain that
$w_1(u)=u$ and for $n\geq 2$,
\begin{equation}\label{rec}
w_n(u)={\bf i} P_{-}\partial_x w_{n-1}(u)+
\sum_{k=2}^n\sum_{\substack{j_1+\cdots+ j_k=n\\j_1,\dots, j_k\geq 1}}c(j_1,\dots,j_k)w_{j_1}(u)\cdots w_{j_k}(u)
\end{equation}
for suitable constants $c(j_1,\dots,j_k)$.
Using a recurrence on $n$, we deduce from \eqref{rec} that $w_n(u)$ is a sum of homogenous expressions of $u$ of order between $1$ and $n$.
Thus we can write
\begin{equation}\label{geo}
w_n(u)=w_n^L(u)+w_n^Q(u)+w_n^C(u)+w^r_n(u)
\end{equation}
where :
\\
\\
$w_n^L(u)$ denotes the terms that appear $w_n(u)$ which are homogeneous of order $1$;\\
$w_n^Q(u)$ denotes the terms that appear $w_n(u)$ which are homogeneous of order $2$;\\
$w_n^C(u)$ denotes the terms that appear $w_n(u)$ which are homogeneous of order $3$;\\
$w_n^r(u)$ denotes the terms that appear $w_n(u)$ which are sums homogeneous terms of order $\geq 4$
(here $L$, $Q$, $C$, $r$ stand for linear, quadratic, cubic and remaining).\\
\\
The content of proposition is related with the structure of $\int w_{2m+4}^C(u) dx$.
\\
We substitute \eqref{geo} in \eqref{rec} to get $w_1^L(u)=u$ and for $n\geq 2$,
$$
w_n^L(u)={\bf i} P_{-}\partial_x w_{n-1}u.
$$
Therefore using that $P_{-}^2=P_{-}$, we obtain that
$$
w_n^L(u)={\bf i}^n P_{-}\partial_x^{n-1} u
=
\frac{1}{2}{\bf i}^n(1-{\bf i}H)\partial_x^{n-1} u
, \quad \forall\, n\geq 2\,.
$$
Next, we study the structure of $w_n^Q(u)$. 
We have $w_1^Q(u)=0$.
We substitute \eqref{geo} in \eqref{rec} and we observe that only $k=2$ contributes to give a quadratic expressions which yields 
\begin{equation}\label{eq:quadratic}
w_n^Q(u)= {\bf i} P_-\partial_x w_{n-1}^Q(u) + \sum_{\substack{j_1+j_2=n\\j_1,j_2\geq 1}}
c(j_1,j_2)w_{j_1}^L(u) w_{j_2}^L(u).
\end{equation}
We now turn to $w_n^C(u)$.
We have that $w_1^C(u)=w_2^C(u)=0$.
For $n\geq 3$, we again substitute \eqref{geo} in \eqref{rec} and we observe that only $k=2,3$ contribute to give a cubic expressions which yields 
\begin{multline}\label{invcub}
w_n^C(u)= {\bf i} P_-\partial_x w_{n-1}^C(u) + 
\sum_{\substack{j_1+j_2=n\\j_1,j_2\geq 1}}c(j_1,j_2)w_{j_1}^L(u) w_{j_2}^Q(u)
\\
+
\sum_{\substack{j_1+j_2+j_3=n\\j_1,j_2,j_3\geq 1}}c(j_1,j_2,j_3)w_{j_1}^L(u) w_{j_2}^L(u)w_{j_3}^L(u)\,.
\end{multline}
We now introduce a notation. We note by $\Lambda^0$ the identity map while for $n\geq 1$, 
the notation $\Lambda^n$ stays for an operator of the form $(c_1+c_2H)\partial_x^n$, where $c_1$, $c_2$ are constants.
Therefore, we may write $w_n^{L}(u)=\Lambda^{n-1}(u)$, $n\geq 1$.
We have the following lemma.
\begin{lem}\label{geo1}
For $n\geq 3$, the expression $\int w_n^C(u)dx$ can be written as a combinations of terms of type
\begin{equation}\label{zvezda}
\int \Lambda^{j_1}(u)\Lambda^{j_2}(u)\Lambda^{j_3}(u) dx
\end{equation}
where $j_1+j_2+j_3=n-3$.
\end{lem}
\begin{proof}
We first notice that $\int {\bf i} P_-\partial_x w_{n-1}^C(u) dx=0$, hence by integration of \eqref{invcub}
we deduce
\begin{multline}\label{idcub}
\int w_n^C(u) dx= 
\sum_{\substack{j_1+j_2=n\\j_1,j_2\geq 1}}c(j_1,j_2) \int w_{j_1}^L(u) w_{j_2}^Q(u) dx
\\
+
\sum_{\substack{j_1+j_2+j_3=n\\j_1,j_2,j_3\geq 1}}c(j_1,j_2,j_3) \int w_{j_1}^L(u) w_{j_2}^L(u)w_{j_3}^L(u) dx\,.
\end{multline}
Since $w_n^{L}(u)=\Lambda^{n-1}(u)$, $n\geq 1$
it follows that the second term on the right hand-side in \eqref{idcub} has 
the claimed structure.
Next we turn to the analysis of the terms $\int w_{j_1}^L(u) w_{j_2}^Q(u) dx$
which are involved on the structure of the first term in the right hand-side of \eqref{idcub}.
For that purpose we invoke the following lemma.
\begin{lem}\label{geo2}
Let $j\geq 0$ and $k\geq 2$. Then the expression
$\int \Lambda^j(u)w_{k}^{Q}(u) dx$
can be written as a combination of terms of type \eqref{zvezda} with $j_1+j_2+j_3=k+j-2$.
\end{lem}
\begin{proof}
We perform an induction on $k$. Since $w_2^Q(u)=c u^2$, we obtain that
$$
\int \Lambda^j(u)w_{2}^{Q}(u) dx= c\int \Lambda^j(u) \Lambda^{0} u \Lambda^{0} u dx\,.
$$
Thus the claim holds for $k=2$.
Next, for $k\geq 3$, we can write
$$
\int\Lambda^j(u)w_{k}^{Q}(u)dx=
\int\Lambda^j(u)
\Big(
{\bf i} P_{-}\partial_x w_{k-1}^Q(u) + \sum_{\substack{j_1+j_2=k\\j_1,j_2\geq 1}}
c(j_1,j_2)w_{j_1}^L(u) w_{j_2}^L(u)
\Big) dx
$$
where we have used \eqref{eq:quadratic}.
Using once again that $w_n^{L}(u)=\Lambda^{n-1}(u)$, $n\geq 1$, we obtain that for $j_1+j_2=k$, the expression
$$
\int\Lambda^j(u)w_{j_1}^L(u) w_{j_2}^L(u) dx
$$
has the claimed structure. It remains to analyze 
$$
\int\Lambda^j(u)P_{-}\partial_x w_{k-1}^Q(u) dx\,.
$$
If we denote by
$P_{+}=\frac{1}{2}(1+{\bf i}H)$ 
the projection on the positive frequencies, we have that
$
\int P_{-}(f)g dx=\int f P_{+}(g) dx
$
and therefore 
$$
\int\Lambda^j(u)P_{-}\partial_x w_{k-1}^Q(u) dx=
-\int P_{+}\partial_{x}\Lambda^j(u) w_{k-1}^Q(u) dx.
$$
The key observation is that $P_{+}\partial_{x}\Lambda^j(u)$ can be written as $\Lambda^{j+1}(u)$ and therefore we are in a position to apply 
the induction hypothesis.
This completes the proof of Lemma~\ref{geo2}.
\end{proof}
Using Lemma~\ref{geo2}, we obtain that for $j_1+j_2=n$, the expression $\int w_{j_1}^L(u) w_{j_2}^Q(u) dx$ can be written as a 
combination of terms of type \eqref{zvezda} with $j_1+j_2+j_3=n-3$.
This completes the proof of Lemma~\ref{geo1}.
\end{proof}
Let us now complete the proof of Proposition~\ref{prop:cubstruc}.
Thanks to Lemma~\ref{geo1} the only terms which can eventually appear in the second terms of the left hand-side of \eqref{even} are 
\begin{multline*}
I=\int u \partial_{x}^{m}(u) \partial_{x}^{m+1}u dx,\quad
II=\int u \partial_{x}^{m}(Hu) \partial_{x}^{m+1}u dx
\\
III=\int u \partial_{x}^{m}u \partial_{x}^{m+1}(Hu)dx ,\quad
IV=\int u \partial_{x}^{m}(Hu) \partial_{x}^{m+1}(Hu)dx.
\end{multline*}
We can write
$$
I=-\frac{1}{2}\int \partial_{x}u(\partial_{x}^m u)^2 dx,\quad
IV=-\frac{1}{2}\int \partial_{x}u(\partial_{x}^m(Hu))^2 dx
$$
and therefore $I$ and $IV$ can be transferred to the third term in the right hand-side of \eqref{even}.
Next, we can write 
$$
III=-II-\int \partial_x u \partial_{x}^{m}u \partial_{x}^{m}(Hu) dx.
$$
The expression $\int \partial_x (u) \partial_{x}^{m}(u) \partial_{x}^{m}(Hu)$ can also be transferred to the third term in the right hand-side of \eqref{even}.
Therefore the expression $II$ is the only one which remains in the second terms of the right hand-side of \eqref{even}. 
This completes the proof of Proposition~\ref{prop:cubstruc}.
\end{proof}
\section{Estimates for $\frac{d}{dt} E_{j/2}\Big (\pi_N \Phi_t^N(u)\Big )_{t=0}$}\label{sec.trun}
For any given $N$ we introduce
the Cauchy problems
\begin{equation}\label{BONN}
\left \{ \begin{array}{c}
\partial_t u_N + H \partial_x^2 u_N + \pi_N \big ((\pi_Nu_N)\partial_x (\pi_Nu_N)\big )=0\\
u(0)=u_0
\end{array}
\right .
\end{equation}
The corresponding unique global solutions (that exist provided that
$u_0\in H^s$ for some $s\geq 0$)
are denoted by
$$u_N(t,.)=\Phi_t^N(u_0)
$$
(see section \ref{detsec} for more details on the truncated problems defined above).
We shall need the following functions (where $k=2(m+1)$ is an even integer as in Theorem~\ref{lem4}):
\begin{equation}\label{eq:G}G_N(u)= \frac{d}{dt} \Big(E_{m+1}(\pi_N(\Phi^t_Nu))\Big )_{t=0}\end{equation}
\begin{equation}\label{eq:H}H_N(u)= \frac{d}{dt} \Big(E_{m+1/2} (\pi_N(\Phi^t_N u))\Big)_{t=0}\end{equation}
\begin{equation}\label{eq:L}L_N^{j_0}(u)= \frac{d}{dt} \Big(E_{j_0/2} 
(\pi_N(\Phi^t_N u))\Big)_{t=0}, j_0=0,...,2m
\end{equation}
defined on the probability space $(H^s, d\mu_{m+1})$ for $s<m+1/2$.
We have the following key property.
\begin{prop}\label{ens}
Let $q\in [1,\infty)$ and $m\geq 2$, then we have:
$$
\lim_{N\rightarrow \infty} 
\Big(\|G_N(u)\|_{L^q(d\mu_{m+1})}+\|H_N(u)\|_{L^q(d\mu_{m+1})}+
\sum_{j_0=0}^{2m}\|L_N^{j_0}(u)\|_{L^q(d\mu_{m+1})}\Big)=0.
$$
\end{prop}
The main tools involved in the proof of Proposition \ref{ens} are in \cite{TV}. More precisely we recall below 
Lemma~9.1 and Lemma~10.1 in \cite{TV}.
\begin{lem}\label{intpar1}
Let $u(x)=\sum_{j=-N}^N c_j e^{{\bf i}jx}$ with $c_0=0$, 
and $u^+(x)=\sum_{j=1}^N c_j e^{{\bf i}jx}$, $u^-(x)=\sum_{j=-N}^{-1} c_j e^{{\bf i}jx}$. 
Then the following identities occur:
\begin{equation}\label{mire1}
\int u (H\partial_x^m \pi_{>N} 
(u\partial_x u)) \partial_x^{m+1} u dx
\end{equation}
$$=\sum_{j=1}^m a_j [\int \pi_{>N} (\partial^j_x u^+
\partial^{m-j+1}_x u^+ )\pi
_{>N} (u^- \partial^{m+1}_xu^-)$$
$$- \pi_{>N} (\partial_x^j 
u^-\partial_x^{m-j+1} u^-)\pi_{>N} 
(u^+ \partial_x^{m+1}u^+) dx]$$
for suitable coefficient $a_j\in \C$;
\begin{equation}\label{mire2}
\int u (H\partial_x^m u) \partial_x^{m+1} \pi_{>N} (u \partial_x u) dx
\end{equation}$$=\sum_{j=1}^m b_j [\int \pi_{>N} 
(\partial_x^j u^+\partial_x^{m-j+1} u^+ )\pi_{>N} 
(u^- \partial_x^{m+1}u^-) $$
$$- \pi_{>N} (\partial^j_x u^-\partial^{m-j+1}_x u^-)\pi_{>N} 
(u^+ \partial^{m+1}_xu^+) dx]$$ for suitable coefficient $b_j\in \C$.
\end{lem}
\begin{lem}\label{prod}
The following estimate occurs :
\begin{equation}\label{double}
\sum_{\substack{|n+m|>N\\0< |n|, |m|\leq N}} 
\frac 1{n^2}\frac 1{|m|}=O\Big (\frac{\ln N}{N}\Big )
\hbox{ as } N\rightarrow \infty.
\end{equation}
\end{lem}
We shall also need a concrete representation of the function
$$\frac{d}{dt} E(\pi_N (u_N(t)))_{|t=0}$$
where $E$ is one of the energies $E_{j/2}$ which are preserved along the flow of \eqref{bo}, and $u_N(t,x)$ are solutions of \eqref{BONN}.\\
In the sequel we shall use the notations introduced in Section~\ref{matconla}. 
Given any $p(u)\in \cup_{n=2}^\infty {\mathcal P}_n(u)$ and any $N\in \N$ then we can introduce
$p_N^*(u)$ as follows (see Section~8 in \cite{TV} for more details). 
Let $p(u) \hbox{ be such that }$ $$\tilde p(u)=\prod_{i=1}^n 
\partial_{x}^{\alpha_i} u$$ 
(see Section~\ref{matconla} for the definition of $\tilde p(u)$) for suitable
$0\leq \alpha_1\leq ... \leq \alpha_n$ and $\alpha_i\in \N$.
First we define $p_{i,N}^*(u)$ as the function 
obtained by $p(u)$ replacing
$\partial_x^{\alpha_i}(u)$ by $\partial_x^{\alpha_i}
(\pi_{>N} (u\partial_x u))$, i.e.
\begin{equation} 
p_{i,N}^*(u) = p(u)_{|\partial_x^{\alpha_i} u= 
\partial_x^{\alpha_i} (\pi_{>N} (u\partial_x u))}, 
\hbox{ } \forall i=1,..,n.
\end{equation}
We now define
$p_N^*(u)$ as follows:
$$p^*_N(u)=\sum_{i=1}^n p_{i,N}^*(u).$$
The following propositions follow by section 8 in \cite{TV}.
\begin{prop}\label{defG}
For every fixed integer $n\in \N$ and for every $N\in \N $ we have:
\begin{align*}
\frac d{dt}E_{n}(\pi_N (u_N(t)))= & 
\sum_{\substack{p(u)\in {\mathcal P}_3(u) s.t. \\ 
\tilde p(u)=u\partial_x^{n-1} u \partial_x^{n}u}}
c_{2n}(p) \int p_N^*(\pi_N (u_N(t)))dx\\
& +
\sum_{\substack{p(u)\in {\mathcal P}_{j}(u) s.t. j=3,..., 2n+2\\ 
\|p(u)\|= 2n-j+2\\ |p(u)|\leq n-1}} c_{2n}(p) \int p_N^*(\pi_N (u_N(t)))dx
\end{align*}
where
$u_N(t,x)$ solves \eqref{BONN} and $c_{2n}(p)$
are the same constants that appear in \eqref{even}.
\end{prop}

\begin{prop}\label{defGodd}
For every integer $n\in \N$ and for every $N\in \N $ we have:
\begin{equation*}
\frac d{dt}E_{n+1/2}(\pi_N(u_N(t)))=  
\sum_{\substack{p(u)\in {\mathcal P}_j(u) s.t.j=3,...,2n+3\\ 
\|p(u)\|=2n-j+3\\|p(u)|\leq n}} c_{2n+1}(p) \int p_N^*(\pi_N (u_N(t)))dx\end{equation*}
where
$u_N(t,x)$ solves \eqref{BONN} and $c_{2n+1}(p)$
are the same constants that appear in \eqref{odd}.
\end{prop}
\begin{proof}[Proof of Proposition~\ref{ens}]
We first prove that 
$\lim_{N\rightarrow \infty} 
\|G_N(u)\|_{L^q(d\mu_{m+1})}=0$.
In fact by combining Proposition \ref{prop:cubstruc} with Proposition \ref{defG} 
it is sufficient to prove
\begin{equation}\label{eq:sing}
\lim_{N\rightarrow \infty} \Big \|\int p^*_N(\pi_N u) dx\Big \|_{L^q(d\mu_{m+1})}=0
\end{equation}
$$(\hbox{ i.e. } \lim_{N\rightarrow \infty} \Big \|\int p^*_N(\pi_N \varphi(\omega)) dx\Big \|_{L^q_\omega}=0
\hbox{ where $\varphi(\omega)=\varphi_{m+1}(\omega)$ is defined in \eqref{randomized}})$$
with:
\begin{equation}\label{eqin1}
p(u)=u (H\partial_x^m u) \partial_x^{m+1}u
\end{equation}
and
\begin{equation}\label{eqin2}
p(u)\in \bigcup_{j=3}^{2m+4} {\mathcal P}_j(u) 
\end{equation}
with
$\|p(u)\|= 2m-j+4$ and $|p(u)|\leq m$.
\\
First we treat the case \eqref{eqin1}.
In this case we can write explicitly
\begin{multline*}
p_N^*(u)= \pi_{>N} (u \partial_x u)(\partial_x^m H u) \partial_x^{m+1}u
\\
+ 
u \partial_x^m (\pi_{>N} H (u\partial_x u)) \partial_x^{m+1}u 
+ u(\partial_x^m H u) \partial_x^{m+1}(\pi_{>N} (u\partial_x u))).
\end{multline*}
Hence we get
$$\int p_N^*(\pi_N(\varphi(\omega)))dx= 
I_N(\omega) + II_N(\omega)$$
where
\begin{equation}\label{INNN}
I_N(\omega)= \int \pi_{>N} 
(\varphi_N(\omega) \partial_x (\varphi_N(\omega)))
\partial_x^m (H \varphi_N(\omega))
\partial_x^{m+1}\varphi_N(\omega) dx
\end{equation}
\begin{equation}\label{INN9}II_N(\omega)= \int 
\varphi_N(\omega) (\partial_x^m H \pi_{>N} (\varphi_N(\omega)
(\partial_x \varphi_N(\omega))) 
\partial_x^{m+1}\varphi_N(\omega)) 
\end{equation}$$+ \varphi_N(\omega)(H \partial_x^m 
\varphi_N(\omega))\partial_x^{m+1}(\pi_{>N} (\varphi_N(\omega)
\partial_x(\varphi_N(\omega)))dx$$
and 
$$\varphi_N(\omega)=\sum_{\substack{n\in \Z \setminus \{0\} \\ -N\leq n\leq N}} 
\frac{\varphi_n(\omega)}{|n|^{m+1}}e^{{\bf i}nx}.$$
In order to estimate
$I_N(\omega)$ notice that 
$$I_N(\omega)=\int \pi_{>N} 
(\varphi_N(\omega)\partial_x \varphi_N(\omega))
(\partial_x^m H \varphi_N(\omega)) \partial_x^{m+1}
\varphi_N(\omega) dx
$$$$= \sum_{\substack{0<|j_1|, |j_2|, |j_3|, |j_4|\leq N\\
|j_1+j_2|>N\\j_1+j_2+j_3+j_4=0}} 
c_{j_1, j_2, j_3, j_4}\frac{\varphi_{j_1}(\omega)}{|j_1|^{m+1}}
\frac{\varphi_{j_2}(\omega)}{|j_2|^m}
\frac{\varphi_{j_3}(\omega)}{|j_3|}\varphi_{j_4}(\omega)$$
where $|c_{j_1, j_2, j_3, j_4}|=1$
and hence by the Minkowski inequality
$$\|I_N(\omega)\|_{L^q_\omega}\leq C 
\sum_{\substack{0<|j_1|, |j_2|, |j_3|, |j_4|\leq N\\|j_1+j_2|>N\\j_1+j_2+j_3+j_4=0}}
\frac 1{|j_1|^{m+1}|j_2|^m|j_3|}$$
$$\leq C \Big (\sum_{0<|j_3|\leq N} \frac 1{|j_3|} \Big)
\Big(\sum_{\substack{0<|j_1|, |j_2|\leq N\\|j_1+j_2|>N}}
\frac 1{|j_1|^{m+1}|j_2|^m} \Big)=O\Big(\frac{\ln^2 N}{N}\Big)$$
where we have used Lemma \ref{prod}.\\
Next we estimate $II_N(\omega)$ (see \eqref{INN9}). By Lemma
\ref{intpar1}
it is sufficient to prove that:
\begin{equation}\label{jzerojenne}
\left \|\int \pi_{>N} (\partial_x^j \varphi_N^-(\omega) 
\partial^{m-j+1}_x\varphi^-_N(\omega))
\pi_{>N} (\partial_x \varphi_N^+(\omega) \partial_x^{m+1} 
\varphi_N^+(\omega)) dx \right \|_{L^q_\omega}=o(1)
\end{equation}
and
\begin{equation}\label{jzerojenne1}
\left \|\int \pi_{>N} (\partial_x^j \varphi_N^+(\omega) 
\partial^{m-j+1}_x\varphi^+_N(\omega))
\pi_{>N} (\partial_x \varphi_N^-(\omega) \partial_x^{m+1} 
\varphi_N^-(\omega)) dx \right \|_{L^q_\omega}=o(1)
\end{equation}
$$\hbox{ as } N\rightarrow \infty, 
\hbox{ } \forall j=1,.., m.$$
Indeed the most delicate cases are $j=1,m$. Moreover \eqref{jzerojenne}
and \eqref{jzerojenne1} can be treated by a similar argument. We shall focus for simplicity on \eqref{jzerojenne}
in the case $j=1$ (the case $j=m$ is similar), i.e.
$$\lim_{N\rightarrow \infty}
\left \|\int \pi_{>N} (\varphi_N^+(\omega)
\partial^{m+1}_x\varphi_N^+(\omega))
\pi_{>N} (\partial_x \varphi_N^-(\omega)\partial_x^{m} 
\varphi^-_N(\omega)) dx\right \|_{L^q_\omega}=0.$$
Notice that we have
$$\int \pi_{>N} (\varphi_N^+(\omega)
\partial^{m+1}_x\varphi_N^+(\omega))
\pi_{>N} (\partial_x \varphi_N^-(\omega)\partial_x^{m} 
\varphi^-_N(\omega)) dx$$$$=\sum_{\substack{0<|j_1|, 
|j_2|, |j_3|, |j_4|\leq N\\
j_1, j_2>0, j_3, j_4<0\\
|j_1+j_2|>N\\j_1+j_2+j_3+j_4=0}} 
\frac{\varphi_{j_1}(\omega)}{|j_1|^{m+1}} 
\varphi_{j_2}(\omega)
\frac{\varphi_{j_3}(\omega)}{|j_3|^{m}} 
\frac{\varphi_{j_4}(\omega)}{|j_4|}$$ and hence
the estimate above follows from the following inequalities:
$$\limsup_{N\rightarrow \infty}
\Big \|\sum_{\substack{0<|j_1|, 
|j_2|, |j_3|, |j_4|\leq N\\
j_1, j_2>0, j_3, j_4<0\\
|j_1+j_2|>N\\j_1+j_2+j_3+j_4=0}} 
\frac{\varphi_{j_1}(\omega)}{|j_1|^{m+1}} 
\varphi_{j_2}(\omega)
\frac{\varphi_{j_3}(\omega)}{|j_3|^{m}} 
\frac{\varphi_{j_4}(\omega)}{|j_4|}
\Big \|_{L^q_\omega}$$
$$\leq \limsup_{N\rightarrow \infty} 
C \sum_{\substack{0<|j_1|,|j_3|, |j_4|
\leq N\\
|j_3+j_4|>N}} \frac 1{|j_1|^{m+1} |j_3|^{m}|j_4|}
$$$$\leq C \limsup_{N\rightarrow \infty} 
\Big (\sum_{0<|j_1|\leq N} 
\frac 1{|j_1|^{m+1}} \Big) 
\sum_{\substack{0<|j_3|, |j_4|\leq N\\
|j_3+j_4|>N}} \frac 1{|j_3|^{m}|j_4|}=O
\Big (\frac{\ln N}{N}
\Big )$$
where we have used Lemma \ref{prod} at the last step (recall that by assumption $m\geq 2$).\\
\\\\
Next we prove \eqref{eq:sing} by assuming \eqref{eqin2}.
In particular we treat the case $p(u)\in {\mathcal P}_3(u)$
with 
$$\|p(u)\|= 2m+1 \hbox{ and } |p(u)|\leq m.$$
We treat for simplicity the case 
$p=\partial_x^\alpha u\partial_x^\beta u \partial_x^{\gamma}u$
with $\sup\{\alpha, \beta, \gamma\}\leq m$, $\alpha+\beta+\gamma=2m+1$
(the same argument works for every $p(u)\in {\mathcal P}_3(u)$ such that 
$\tilde p(u)= \partial_x^\alpha u\partial_x^\beta u \partial_x^{\gamma}u$).
Hence we get
$$p^*_N(\varphi_N(\omega))=I_N(\omega) + II_N(\omega)+III_N(\omega)$$
where
$$I_N(\omega)= \int \partial_x^\alpha(\pi_{>N} 
(\varphi_N(\omega) \partial_x \varphi_N(\omega)))
\partial_x^\beta \varphi_N(\omega) \partial_x^{\gamma} \varphi_N(\omega) dx,
$$
$$II_N(\omega)= \int 
\partial_x^\alpha \varphi_N(\omega) \partial_x^\beta 
(\pi_{>N} (\varphi_N(\omega)\partial_x \varphi_N(\omega))) 
\partial_x^{\gamma} \varphi_N(\omega) dx,
$$$$III_N(\omega)= \int \partial_x^\alpha 
\varphi_N(\omega)\partial_x^\beta  
\varphi_N(\omega)\partial^{\gamma}_x
(\pi_{>N} (\varphi_N(\omega)
\partial_x \varphi_N(\omega))) dx.$$ 
We shall prove that
$$\lim_{N\rightarrow \infty} \|I_N(\omega)\|_{L^q_\omega}=0$$
(and in a similar way we can treat $II_N(\omega)$ 
and $III_N(\omega)$). By 
the Leibnitz formula 
it is sufficient to prove 
$$
\lim_{N\rightarrow \infty}
\Big \|\int \pi_{>N} (\partial_x^j \varphi_N
(\omega) \partial_x^{\alpha-j+1}\varphi_N(\omega))
\partial_x^\beta \varphi_N(\omega) \partial_x^{\gamma} 
\varphi_N(\omega) dx\Big \|_{L^q_\omega}=0$$
$$\forall j=0,...,\alpha.$$
We shall treat the case $j=0$ and all the other 
cases can be treated in a similar way.
More precisely we shall prove that
$$
\lim_{N\rightarrow \infty}
\Big \|\int \pi_{>N} (\varphi_N(\omega) 
\partial_x^{\alpha+1}\varphi_N(\omega))
\partial_x^\beta \varphi_N(\omega) \partial_x^{\gamma} 
\varphi_N(\omega) dx\Big \|_{L^q_\omega}=0.$$
Notice that we have 
$$\int \pi_{>N} (\varphi_N(\omega) 
\partial_x^{\alpha+1}\varphi_N(\omega))
\partial_x^\beta \varphi_N(\omega) \partial_x^{\gamma} 
\varphi_N(\omega) dx$$$$=
\sum_{\substack{|j_1|, |j_2|, |j_3|, |j_4|\in (0,N],\\
|j_1+j_2|>N\\j_1+j_2+j_3+j_4=0}} \frac{\varphi_{j_1}(\omega)}
{|j_1|^{m+1}} \frac{\varphi_{j_2}(\omega)}{|j_2|^{m-\alpha}}
\frac{\varphi_{j_3}(\omega)}{|j_3|^{m+1-\beta}} 
\frac{\varphi_{j_4}(\omega)}{|j_4|^{m+1-\gamma}}$$
and hence
by using the triangular inequality we get
$$\Big \|\int \pi_{>N} (\varphi_N(\omega) 
\partial_x^{\alpha+1}\varphi_N(\omega))
\partial_x^\beta \varphi_N(\omega) \partial_x^{\gamma} 
\varphi_N(\omega) dx\Big \|_{L^q_\omega}$$$$\leq C
\sum_{\substack{|j_1|, |j_2|, |j_3|, |j_4|\in (0,N],\\
|j_1+j_2|>N\\j_1+j_2+j_3+j_4=0}} \frac 1{|j_1|^{m+1}
|j_2|^{m-\alpha} 
|j_3|^{m+1-\beta} |j_4|^{m+1-\gamma}}.$$
Next we consider three possible cases:
\\
\\
{\em First subcase: $\alpha=1$, $\beta=\gamma=m$} 
\\
\\
In this case
we get
$$\Big \|\int (\pi_{>N} \varphi_N(\omega) 
\partial_x^{\alpha+1}\varphi_N(\omega))
\partial_x^\beta \varphi_N(\omega) 
\partial_x^{\gamma} \varphi_N(\omega) dx\Big \|_{L^q_\omega}$$
$$\leq C
\sum_{\substack{|j_1|, |j_2|, |j_4|\in (0,N],\\|j_1+j_2|>N}} 
\frac 1{|j_1|^{m+1}|j_2|^{m-1} |j_4|}$$
$$\leq \Big(\sum_{0<|j_4|\leq N} 
\frac 1{|j_4|}\Big)\Big(\sum_{\substack{0<|j_1|, |j_2|\leq N,\\ |j_1+j_2|>N}}
\frac 1{|j_1|^{m+1}|j_2|^{m-1}}\Big )
=O\Big (\frac{\ln ^2 N}{N} \Big) $$
where we have used Lemma \ref{prod}. 
\\
\\
{\em Second subcase: $\alpha\leq \beta=\gamma< m$} 
\\
\\
In this case we get
$$\Big \|\int (\pi_{>N} \varphi_N(\omega) 
\partial_x^{\alpha+1}\varphi_N(\omega))
\partial_x^\beta \varphi_N(\omega) \partial_x^{\gamma} 
\varphi_N(\omega) dx\Big \|_{L^q_\omega}$$$$\leq C
\sum_{\substack{|j_1|, |j_2|, |j_3|, |j_4|\in (0,N],\\
|j_1+j_2|>N}} \frac 1{|j_1|^{m+1}|j_2| 
|j_3|^{2}}$$
$$\leq C \Big(\sum_{0<|j_4|\leq N} 
\frac 1{|j_3|^2}\Big)\Big(\sum_{\substack{0<|j_1|, |j_2|\leq N,\\ |j_1+j_2|>N}}
\frac 1{|j_1|^{m+1}|j_2|}\Big )
=O\Big (\frac{\ln N}{N} \Big) $$
where we have used Lemma \ref{prod}.
\\
\\
{\em Third subcase: $\alpha\leq \beta<\gamma\leq m$} 
\\
\\
In this case we get
$$\Big \|\int (\pi_{>N} \varphi_N(\omega) 
\partial_x^{\alpha+1}\varphi_N(\omega))
\partial_x^\beta \varphi_N(\omega) \partial_x^{\gamma} 
\varphi_N(\omega) dx\Big \|_{L^q_\omega}$$
$$\leq C
\sum_{\substack{|j_1|, |j_2|, |j_3|, |j_4|\in (0,N],\\
|j_1+j_2|>N}} \frac 1{|j_1|^{m+1}|j_2| 
|j_3|^{2}}$$
and we can conclude as in the previous case. 
\\\\
The proof of \eqref{eq:sing} under the assumption 
$p(u)\in {\mathcal P}_j(u)$ with $j=4,...,2m+4$ and \begin{equation*}
\|p(u)\|= 2m-j+4 \hbox{ and } |p(u)|\leq m
\end{equation*}
can be done by a similar argument as above.
\\
\\
By Propositions \ref{defG} and \ref{defGodd} the proof of 
$$\lim_{N\rightarrow \infty}\|H_N(u)\|_{L^q(d\mu_{m+1})}+
\sum_{j_0=0}^{2m}\|L_N^{j_0}(u)\|_{L^q(d\mu_{m+1})}=0$$
follows from 
\begin{equation*}
\lim_{N\rightarrow \infty} \Big \|\int p^*_N(\pi_N u) dx\Big \|_{L^q(d\mu_{m+1})}=0
 \end{equation*}
where:
\begin{equation*}
p(u)\in \bigcup_{j=3}^{2m+3} {\mathcal P}_j(u), 
\end{equation*}
such that
$\|p(u)\|\leq 2m$ and  $ |p(u)|\leq m$.
Those estimates can be done arguing as in the proof of 
\eqref{eq:sing} under the assumption \eqref{eqin2}. We skip the details.
This completes the proof of Proposition~\ref{ens}.
\end{proof}
\section{Some deterministic results}\label{detsec}
We shall study qualitative properties of solutions to the following Cauchy problems:
\begin{equation}\label{BOdet}
\left \{ \begin{array}{c}
\partial_t u + H \partial_x^2 u + u\partial_x u=0\\
u(0)=u_0
\end{array}
\right .
\end{equation}
and (for every fixed $N\in \N$)
\begin{equation}\label{BON}
\left \{ \begin{array}{c}
\partial_t u_N + H \partial_x^2 u_N + \pi_N \big ((\pi_Nu_N)\partial_x (\pi_Nu_N)\big )=0\\
u(0)=u_0
\end{array}
\right .
\end{equation}
The corresponding unique global solutions (that exist provided that
$u_0\in H^s$ for some $s\geq 0$)
are denoted respectively by
$$u(t, .)=\Phi_t(u_0)
\hbox{ and }
u_N(t,.)=\Phi_t^N(u_0).
$$
Indeed, in the case of \eqref{BON}, to get the global well-posedness 
one simply needs to use that the frequencies $>N$ evolve linearly,
while the other frequencies evolve under an ODE with a conserved
$L^2$ norm. 
For every subset $A\subset H^s$ (with $s\geq 0$ fixed) and for every $t\in \R$ we
define the set $\Phi_t^N (A)$ as follows:
\begin{equation}\label{eq:PhiNBO}
\Phi_t^N(A)=\{u_N (t,.)\in H^s|  \hbox{ where }
u_N(t,.) \hbox{ solves } \eqref{BON} \hbox{ with } u_0\in A\}.
\end{equation}
Recall that the definition of $\Phi_t(A)$ is given in \eqref{eq:PhiBO}.
The main result of this section is the following proposition.
\begin{prop}\label{prop:det}
Let $2\leq s<\sigma$ be fixed and $R>0$. Then there exists 
$\bar t=\bar t(R)>0$ such that 
for every $\varepsilon>0$ there exists $N_0(\varepsilon)$ with the property
$$\Phi_{t}^N(A)\subset \Phi_{t}(A) + B^s(\varepsilon), \hbox{ } \forall N>N_0(\varepsilon),
\forall t\in (-\bar t, \bar t), \forall A\subset B^\sigma(R).$$
\end{prop}
First we prove some lemmas.
\begin{lem}\label{lem:boundH2}
Let $R>0$ and $T>0$ be fixed, then
$$ 
\sup_{\substack{t\in [0,T]\\u_0\in B^\sigma(R)}} 
\|\Phi_t(u_0)\|_{H^2}<\infty.
$$ 
\end{lem}
\begin{proof} The proof is standard and follows from the conservation of
$E_0, E_{1/2}, E_1,E_{3/2}, E_2$ along the
solutions of \eqref{BOdet}.
We skip the details.
\end{proof}
\begin{lem}\label{lem:boundedness}
Let $\sigma>2$, $T>0$ be fixed and $R>0$, then
\begin{equation}\label{eq:ni}\sup_{\substack{t\in [0,T]\\u_0\in B^\sigma(R)}} \|\Phi_t(u_0))\|_{H^\sigma}<\infty.
\end{equation}
Moreover there exists $\bar t=\bar t(R)\in (0, T]$ such that
\begin{equation}\label{eq:nik}\sup_{\substack{t\in [0,\bar t]\\N\in \N, u_0\in B^\sigma(R)}} \|\Phi_t^N(u_0))\|_{H^\sigma}<\infty.\end{equation}
\end{lem}
\begin{proof}
{\bf First step: estimate for 
$\Phi_t(u_0)$ (uniform in time)}\\
\\
Set $D=(1-\partial_x^2)^{1/2}$. We have
$$\partial_t (D^\sigma \Phi_t(u_0)) + H \partial_x^2 (D^\sigma \Phi_t(u_0)) 
+ D^\sigma( \Phi_t(u_0)\partial_x \Phi_t(u_0))=0.
$$
Multiplication by $D^\sigma \Phi_t(u_0)$  in conjunction with standard properties of the Hilbert transform $H$
and with elementary calculus gives
\begin{equation}\label{eq:gron1}\frac 12 \frac d{dt} \|\Phi_t(u_0)\|_{H^\sigma}^2+ 
\int D^\sigma(\Phi_t(u_0)\partial_x \Phi_t(u_0)) D^\sigma \Phi_t(u_0) dx=0.
\end{equation}
Notice that we have the following identity 
\begin{multline}\label{kato}
\int D^\sigma(\Phi_t(u_0)\partial_x \Phi_t(u_0)) D^\sigma \Phi_t(u_0) dx=
\\
\int \Phi_t(u_0)\partial_x 
(D^\sigma \Phi_t(u_0))D^\sigma \Phi_t(u_0)
+\int [D^\sigma, \Phi_t(u_0)] \partial_x \Phi_t(u_0) D^\sigma \Phi_t(u_0) dx.
\end{multline}
By using integration by parts and the Sobolev embedding $H^1\subset L^\infty$, we estimate the first term on the r.h.s. of \eqref{kato} as follows:
$$
\Big|\int \Phi_t(u_0)\partial_x 
(D^\sigma \Phi_t(u_0))D^\sigma \Phi_t(u_0)\Big|
\leq C \| \Phi_t(u_0)\|_{H^2} \|\Phi_t(u_0)\|_{H^\sigma}^2.
$$
Next, we recall the following form of the Kato-Ponce (see \cite{Ponce}) commutator estimate:
\begin{equation}\label{gger}
\|[D^\sigma, f] g\|_{L^2}\leq C(\|f\|_{H^2}\|g\|_{H^{\sigma-1}}+\|f\|_{H^\sigma}\|g\|_{H^1}). 
\end{equation}
Estimate \eqref{gger} is obtained in \cite{Ponce} for functions on $\R$. Its extension to periodic functions can be done by a localization argument.
By combining \eqref{gger} with the Cauchy-Schwarz inequality, we can estimate the second 
term on the r.h.s. of \eqref{kato} as follows:
$$
\Big|\int [D^\sigma, \Phi_t(u_0)] \partial_x \Phi_t(u_0) D^\sigma \Phi_t(u_0) dx\Big|
\leq C \| \Phi_t(u_0)\|_{H^2} \|\Phi_t(u_0)\|_{H^\sigma}^2.
$$
Therefore, we obtained the estimate 
$$\Big| \int D^\sigma(\Phi_t(u_0)\partial_x \Phi_t(u_0)) D^\sigma \Phi_t(u_0) dx\Big|
\leq C \| \Phi_t(u_0)\|_{H^2} \|\Phi_t(u_0)\|_{H^\sigma}^2.$$
Hence by Lemma~\ref{lem:boundH2} and \eqref{eq:gron1}
we get
$$\Big |\frac 12 \frac d{dt} \|\Phi_t(u_0)\|_{H^\sigma}^2\Big |\leq C \|\Phi_t(u_0)\|_{H^\sigma}^2,
\hbox{ } \forall u_0\in  B^\sigma(R)$$
that by the Gronwall lemma gives 
$$\sup_{\substack{t\in [0,T]\\u_0\in  B^\sigma(R)}} \|\Phi_t (u_0)\|_{H^\sigma}<\infty.$$
This concludes the proof of \eqref{eq:ni}.
\\
{\bf Second step: estimate for $\Phi^N_t(u_0)$ (for short time)}
\\
\\
Notice that the solution $u_N(t,x)=\Phi^N_t(u_0)$ to \eqref{BON} can be split as
$$u_N(t, x)=v_N(t,x)+ w_N(t,x)$$
where $w_N(t,x)$ is the solution of the linear Cauchy problem \begin{equation*}
\left \{ \begin{array}{c}
\partial_t w_N + H\partial_x^2 w_N=0\\
w_N(0)=\pi_{>N} u_0
\end{array}
\right .
\end{equation*}
and $v_N(t,x)$ satisfies the ODE
\begin{equation*}
\left \{ \begin{array}{c}
\partial_t v_N+ H\partial_x^2 v_N+\pi_N (v_N\partial_x v_N)=0\\
v_N(0)=\pi_{N} u_0
\end{array}
\right .
\end{equation*}
Observe that $\pi_N (v_N)=v_N$. Of course the $H^\sigma$-norm is preserved along free evolution.
Hence we have to control just the $H^\sigma$-norm of $v_N(t, x)$ as long as $u_0\in  B^\sigma(R)$.\\
It is useful to introduce the modified flow
\begin{equation}\label{fllow}
\tilde \Phi_t^N(u_0)=v_N(t,x)
\end{equation}
where $v_N(t, x)$ is defined as above.\\
By using the property $[D^\sigma, \pi_N]=0$ we get
$$\partial_t (D^\sigma \tilde \Phi^N_t(u_0)) + H \partial_x^2 (D^\sigma \tilde \Phi^N_t(u_0)) 
+ \pi_N D^\sigma( \tilde \Phi^N_t(u_0)\partial_x \tilde \Phi^N_t(u_0))=0.$$
After multiplication by $D^\sigma \tilde \Phi^N_t(u_0)$ and integration we deduce
\begin{equation*}\frac 12 \frac d{dt} \|\tilde \Phi^N_t(u_0)\|_{H^\sigma}^2+ \int \pi_N D^\sigma(\tilde \Phi^N_t(u_0) \partial_x \tilde \Phi^N_t(u_0) ) 
D^\sigma \tilde \Phi^N_t(u_0) dx=0.
\end{equation*}
Since $\pi_N (\tilde \Phi^N_t(u_0))= \tilde \Phi^N_t(u_0)$, then the identity above is equivalent to
\begin{equation*}\frac 12 \frac d{dt} \| \tilde \Phi^N_t(u_0)\|_{H^\sigma}^2+ \int D^\sigma(\tilde \Phi^N_t(u_0) \partial_x \tilde \Phi^N_t(u_0) ) 
D^\sigma \tilde \Phi^N_t(u_0) dx=0.
\end{equation*}
Arguing as in the first step we get
\begin{equation*}\frac 12 \frac d{dt} \| \tilde \Phi^N_t(u_0)\|_{H^\sigma}^2\leq C  \| \tilde \Phi^N_t(u_0)\|_{H^\sigma}^3
\end{equation*}
which in turn is equivalent to
\begin{equation*}\frac d{dt} \| \tilde \Phi^N_t(u_0)\|_{H^\sigma}\leq C  \| \tilde \Phi^N_t(u_0)\|_{H^\sigma}^2.
\end{equation*}
By the estimate above we deduce  
$$ \| \tilde \Phi^N_t(u_0)\|_{H^\sigma}\leq  \| \pi_N(u_0)\|_{H^\sigma} + C\int_0^t  \| \tilde \Phi^N_s(u_0)\|_{H^\sigma}^2 ds$$
$$\leq  R  + C\int_0^t  \| \tilde \Phi^N_s(u_0)\|_{H^\sigma}^2 ds, \hbox{ } \forall u_0\in  B^\sigma(R)$$
that in turn implies
$$\sup_{s\in [0, t]}\| \tilde \Phi^N_s(u_0)\|_{H^\sigma} \leq  R + C t 
\Big (\sup_{s\in [0, t]}\| \tilde \Phi^N_s(u_0)\|_{H^\sigma}\big )^2, \hbox{ } \forall u_0\in  B^\sigma(R).$$
Next we consider the real valued function
$$x\rightarrow f_{R,t}(x)=x - R -Ctx^2$$
and we notice that if we denote by
$x_{\pm}(R, t)$ the solutions of  $f_{R,t}(x)=0$,
then 
$$x_{\pm}(R, \bar t )\in \R,x_{-}(R, \bar t)<x_{+}(R, \bar t)  \hbox{ and } x_{-}(R, \bar t)= 4R $$$$\hbox{ provided that } \bar t= 3/(16CR).$$ 
The conclusion follows by a classical continuity argument in conjunction with the fact that
the function 
$$t\rightarrow F_{u_0,N}(t)=\sup_{s\in [0, t]}\| \tilde \Phi^N_s(u_0)\|_{H^\sigma}$$
is continuous and  $F_{u_0,N}(0)\in [0, R]$.

\end{proof}
\begin{proof}[Proof of Proposition~\ref{prop:det}]
We give the proof only for positive times. The analysis for negative times is the same, modulo some direct modifications.
We claim the following estimate
\begin{equation}\label{eq:claim}
\lim_{N\rightarrow \infty} \Big (\sup_{\substack{t\in [0,\bar t]\\u_0\in A}} 
\|\Phi_t(u_0)- \Phi^N_t(u_0)\|_{L^2}\Big )=0
\end{equation}
where $\bar t=\bar t(R)$ is given in Lemma \ref{lem:boundedness}.
Notice that by interpolation we get
$$\|\Phi_t(u_0)- \Phi^N_t(u_0)\|_{H^s}\leq \|\Phi_t(u_0)- \Phi^N_t(u_0)\|_{L^2}^\theta 
\|\Phi_t(u_0)- \Phi^N_t(u_0)\|_{H^\sigma}^{1-\theta}
$$
for a suitable $\theta\in (0, 1)$.
By combining this fact with \eqref{eq:claim} and with  Lemma \ref{lem:boundedness}
we get
$$\lim_{N\rightarrow \infty} \Big (\sup_{\substack{t\in [0,\bar t]\\u_0\in A}} 
\|\Phi_t(u_0)- \Phi^N_t(u_0)\|_{H^s}\Big )=0$$
which concludes the proof of Proposition~\ref{prop:det}.\\
Next we focus on the proof of \eqref{eq:claim}.
Notice that
$\Phi_t(u_0)- \Phi^N_t(u_0)$ solve the following equation
$$\partial_t (\Phi_t(u_0)- \Phi^N_t(u_0)) + H\partial_x^2 (\Phi_t(u_0)- \Phi^N_t(u_0)) $$$$+ \frac 12 \partial_x ((\Phi_t(u_0))^2- 
(\pi_N \Phi_t^N(u_0))^2)
+ \frac 12 (1-\pi_N)\partial_x (\pi_N \Phi^N_t(u_0))^2=0.$$
Multiplication by $\Phi_t(u_0)- \Phi^N_t(u_0)$ and integration give:
\begin{equation}\label{eq:claim1}\frac 12 \frac d{dt} \int (\Phi_t(u_0)- \Phi^N_t(u_0))^2 dx\end{equation}
$$+ \frac 12 \int \partial_x ((\Phi_t(u_0))^2-(\pi_N \Phi^N_t(u_0))^2) (\Phi_t(u_0)- \Phi^N_t(u_0))
$$$$+ \frac 12 \int (\Phi_t(u_0)- \Phi^N_t(u_0)) \pi_{>N} \partial_x (\pi_N \Phi^N_t(u_0))^2dx=0.$$
By integration by parts we get
$$\int \partial_x ((\Phi_t(u_0))^2-(\pi_N \Phi^N_t(u_0))^2) (\Phi_t(u_0)- \Phi^N_t(u_0)) dx$$$$
= - \int((\Phi_t(u_0))^2-(\pi_N \Phi^N_t(u_0))^2) \partial_x (\Phi_t(u_0)- \Phi^N_t(u_0)) dx$$
$$= - \int((\Phi_t(u_0))^2-(\pi_N \Phi^N_t(u_0))^2) \partial_x (\Phi_t(u_0)- \pi_N \Phi^N_t(u_0)) 
dx $$$$+ \int((\Phi_t(u_0))^2-(\pi_N \Phi^N_t(u_0))^2) \partial_x \big(\pi_{>N} \Phi^N_t(u_0)\big) dx
$$
$$=
\frac 12 \int (\partial_x \Phi_t(u_0) + \partial_x \big(\pi_N \Phi^N_t(u_0)\big)) 
(\Phi_t(u_0)- \pi_N \Phi^N_t(u_0))^2 dx
$$$$+ \int((\Phi_t(u_0))^2-(\pi_N \Phi^N_t(u_0))^2) \partial_x \big(\pi_{>N} \Phi^N_t(u_0)\big) dx
$$
and hence by the H\"older inequality
\begin{equation*} \Big |\int \partial_x 
((\Phi_t(u_0))^2-(\pi_N \Phi^N_t(u_0))^2) (\Phi_t(u_0)- \Phi^N_t(u_0)) dx\Big |\end{equation*}
\begin{equation*}\leq \frac 12 (\|\partial_x \Phi_t(u_0)\|_{L^\infty} + \|\partial_x \big(\pi_N \Phi^N_t(u_0)\big)
\|_{L^\infty}) \|\Phi_t(u_0)- \pi_N \Phi^N_t(u_0)\|_{L^2}^2
\end{equation*}
$$
+ (\|\Phi_t(u_0)\|_{L^\infty} + \|\pi_N \Phi^N_t(u_0)
\|_{L^\infty}) \|\Phi_t(u_0)- \Phi^N_t(u_0)\|_{L^2}\|\big(\pi_{>N} \Phi^N_t(u_0)\|_{H^1}$$$$\forall t\in [0, \bar t].
$$
By the Sobolev embedding $H^1\subset L^\infty$ we can continue the inequality as follows
\begin{equation}\label{eq:claim2} \Big |\int \partial_x 
((\Phi_t(u_0))^2-(\pi_N \Phi^N_t(u_0))^2) (\Phi_t(u_0)- \Phi^N_t(u_0)) dx\Big |
\end{equation}
\begin{equation*}
\leq C \|\Phi_t(u_0)- \Phi^N_t(u_0)\|_{L^2}^2
\end{equation*}
$$
+ (\|\Phi_t(u_0)\|_{H^1} + \|\pi_N \Phi^N_t(u_0)
\|_{H^1}) \|\Phi_t(u_0)- \Phi^N_t(u_0)\|_{L^2} N^{-1}\|\Phi^N_t(u_0)\|_{H^2}$$
$$\leq CN^{-2} + \|\Phi_t(u_0)- \Phi^N_t(u_0)\|_{L^2}^2 , \hbox{ } \forall t\in [0, \bar t]$$
where we have used \eqref{eq:ni} and \eqref{eq:nik} in Lemma \ref{lem:boundedness} to control
$$\sup_{\substack{t\in [0,\bar t],\\ u_0\in A}}\big 
\{\|\Phi_t(u_0)\|_{H^1}, \|\Phi^N_t(u_0)\|_{H^2}\big \}<\infty.$$
Moreover by the Cauchy-Schwarz inequality we have the estimate
$$\Big| \int (\Phi_t(u_0)- \Phi^N_t(u_0)) \pi_{>N} \partial_x (\pi_N \Phi^N_t(u_0))^2dx\Big |
$$$$\leq \|\Phi_t(u_0)- \Phi^N_t(u_0)\|_{L^2} \|\pi_{>N} \partial_x (\pi_N \Phi^N_t(u_0))^2\|_{L^2}
$$
$$\leq \|\Phi_t(u_0)- \Phi^N_t(u_0)\|_{L^2} \|\pi_{>N} (\pi_N \Phi^N_t(u_0))^2\|_{H^1}$$
and hence
$$\Big| \int (\Phi_t(u_0)- \Phi^N_t(u_0)) \pi_{>N} \partial_x (\pi_N \Phi^N_t(u_0))^2dx\Big |
$$$$\leq \|\Phi_t(u_0)- \Phi^N_t(u_0)\|_{L^2} N^{-1}\|(\Phi^N_t(u_0))^2\|_{H^2}.$$
Since $H^2$ is an algebra we get
\begin{equation}\label{eq:claim3}
\Big| \int (\Phi_t(u_0)- \Phi^N_t(u_0)) \pi_{>N} \partial_x (\pi_N \Phi^N_t(u_0))^2dx\Big |
\end{equation}
\begin{equation*}
\leq C \|\Phi_t(u_0)- \Phi^N_t(u_0)\|_{L^2} N^{-1}\|\Phi^N_t(u_0)\|^2_{H^2}\end{equation*}\begin{equation*}
\leq CN^{-2} + \|\Phi_t(u_0)- \Phi^N_t(u_0)\|_{L^2}^2 , \hbox{ } \forall t\in [0, \bar t]
\end{equation*}
where we have used \eqref{eq:nik} in Lemma \ref{lem:boundedness} to control
$\sup_{\substack{t\in [0,\bar t],\\ u_0\in A}}\|\Phi^N_t(u_0)\|^2_{H^2}$.
The proof of \eqref{eq:claim} follows by combining \eqref{eq:claim1}, \eqref{eq:claim2},
\eqref{eq:claim3} with the Gronwall lemma (recall that $\Phi_0(u_0)- \Phi^N_0(u_0)=0)$).
 \end{proof}
\section{Proof of Theorem \ref{lem4}}
To simplify the notations we shall denote
$d\mu=d\mu_{k/2}$, $F_N=F_{k/2,N,R}$. In the sequel we shall always assume
that
$2\leq s < \sigma<(k-1)/2$. Since by assumptions $k\geq 6$ is an even number we can introduce $m\geq 2$ such that $k=2(m+1)$.
We also denote by ${\mathcal B}(H^\sigma)$ the Borel sets in $H^\sigma$.
We shall use the Hamiltonian structure of the flow $\tilde{\Phi}^N_t$ and the finite dimensional Liouville theorem
on the invariance of the Lebesgue measure.
For every $N$, we denote by $E_N$ the real vector space spanned by $(\cos(nx),\sin(nx))_{1\leq n\leq N}$. 
From now on, we consider $\tilde{\Phi}^N_t$ as a flow on $E_N$, defined as the restriction of the flow defined by \eqref{fllow} to $E_N$.
We denote by $E_N^\perp$ the orthogonal 
complementary of $E_N$ in $H^\sigma$. We can see the measure 
 $d\mu$ as a product measure on $E_N\times E_N^\perp$ as follows
$$d
\mu= \gamma_N e^{-\|\pi_N u\|_{H^{k/2}}^2} du_1...du_N \times d\mu^\perp_N$$
where $\gamma_N$ is a suitable renormalization factor.
The measure 
$$
\gamma_N e^{-\|\pi_N u\|_{H^{k/2}}^2} du_1...du_N 
$$
is a measure on $E_N$ while $d\mu^\perp_N$ is a measure on $E_N^\perp$.
More precisely
$$
 du_1...du_N \equiv \prod_{n=1}^N d(2a_n)\, d(2b_n),
$$
where $u_n=a_n+ib_n$, $(a_n,b_n)\in \R^2$ and
$$
\pi_N u=\sum_{0<|n|\leq N}u_ne^{inx},\qquad \overline{u_n}=u_{-n}.
$$
We have the following statement.
\begin{prop}\label{ch-var} One has the identity:
\begin{multline*}
\gamma_{N}^{-1}\int_{\Phi^N_t(A)} F_N (u) d\mu=
\int_{A} \prod_{j=0}^{k-2} \chi_R (E_{j/2}(\pi_N \Phi^N_t(u))) \times
\\
\chi_R (E_{(k-1)/2}(\pi_N \Phi^N_t(u))-\alpha_N)
  e^{-E_{k/2}(\pi_N(\Phi^N_t(u))} du_1...du_N\times d\mu^\perp_N.
  \end{multline*}
\end{prop}
\begin{proof}
We need the following two lemmas.
\begin{lem}\label{esp1}
The map $\tilde{\Phi}^N_t$ is measure preserving on $E_N$ equipped with the Lebesgue measure $du_1...du_N$.
\end{lem}
\begin{proof}
This is a consequence of the Liouville theorem, thanks to the hamiltonian structure of the ordinary differential equation defining the flow $\tilde{\Phi}^N_t$.
\end{proof}
\begin{lem}\label{esp2}
The map $S(t)=e^{-tH\partial_x^2}$ is measure preserving on $E_N^\perp$ equipped with the gaussian measure $ d\mu^\perp_N$. 
\end{lem}
\begin{proof}
This claim reflects the invariance of the gaussians distributions on $\R^2$ by rotations.
For a similar analysis, we refer to \cite[Proposition~2.10]{S} (which in turn follows the arguments in \cite[Theorem~1.2]{tz_fourier}).
First of all, clearly $E_N^\perp$ is invariant by $S(t)$.
For $M>N$, we denote by $E_N^M$ the finite dimensional real vector space spanned by $(\cos(nx),\sin(nx))$, where $N<n\leq M$.
We denote by $\mu_N^M$ the centered gaussian measure on $E_N^M$ induced by the series 
$$
\sum_{n=N+1}^M \frac{\varphi_n(\omega)}{|n|^{k/2}} e^{{\bf i}nx}\,.
$$
For $U$ an open set of $E^N$, we have
\begin{equation}\label{fatou}
\mu^\perp_N(U)\leq \liminf_{M\rightarrow\infty}\mu_N^M(U\cap E^M_N)\,.
\end{equation}
Indeed, for $M>N$, we set
$
U^M\equiv(u\in E_N^\perp \,|\, \pi_M u\in U).
$
Then using that $U$ is an open set, we get
$$
U\subset \liminf_{M\rightarrow\infty}(U^M)=\bigcup_{M=1}^{\infty} \bigcap_{M_1=M}^{\infty}U^{M_1}
$$
and therefore
$
\11(U)\leq \liminf_{M\rightarrow\infty}\11(U^M)\,,
$
where $\11$ denotes the indicator function of a set.
On the other hand 
$$
\mu_N^M(U\cap E^M_N)=\int_{E_N^\perp}\11(U^M)d\mu^\perp_N\,.
$$
Now,  \eqref{fatou} follows by an application of Fatou's lemma.
By passing to a complementary set in \eqref{fatou}, we get that for $F$ a closed set of $E^N$,
\begin{equation}\label{fatou-bis}
\mu^\perp_N(F)\geq \liminf_{M\rightarrow\infty}\mu_N^M(F\cap E^M_N)\,.
\end{equation}
Using that $H(\cos(nx))=  \sin(nx)$ and $H(\sin(nx))=-\cos(nx)$, we get
\begin{eqnarray*}
S(t)(\cos(nx)) &  =  &\cos(-tn^2+nx)  =  \cos(tn^2)\cos(nx)+\sin(tn^2)\sin(nx),
\\
S(t)(\sin(nx))   & =  &\sin(-tn^2+nx)  =  -\sin(tn^2)\cos(nx)+\cos(tn^2)\sin(nx).
\end{eqnarray*}
Therefore for fixed $t$ and $n$ the map $S(t)$ acts as a rotation on the two dimensional real vector space spanned by $\cos(nx)$ and $\sin(nx)$.
Hence by the invariance of the Lebesgue measure and the diagonal quadratic forms by rotations, any centered gaussian measure on the two dimensional space ${\rm span}(\cos(nx),\sin(nx))$ is invariant by $S(t)$. This implies that that the measure $\mu_N^M$ (which is a product of such measures) is invariant by $S(t)$.

Let $F$ be a closed set of $E_N^\perp$.  Then $S(t)(F)$ is also closed and thanks to \eqref{fatou-bis},
$$
\mu^\perp_N(S(t)(F)+\overline{B_{\varepsilon}})\geq \limsup_{M\rightarrow\infty}\mu_N^M((S(t)F+\overline{B_{\varepsilon}})\cap E^M_N),
$$
where $B_{\varepsilon}$ denotes the open ball of radius $\varepsilon$ in $E_N^\perp$ (recall that $E_N^\perp$ is equipped with the $H^\sigma$ topology). 
Since $S(t)$ acts as an isometry on $H^\sigma$ and since $E^M_N$ is invariant under $S(t)$,  for every $\varepsilon$ and every $M$,
$$
S(t)\big((F+B_{\varepsilon})\cap E^M_N\big)\subset (S(t)F+\overline{B_{\varepsilon}})\cap E^M_N.
$$
Therefore using the invariance of $\mu_N^M$ by $S(t)$ and \eqref{fatou}, we get
\begin{eqnarray*}
\mu^\perp_N(S(t)(F)+\overline{B_{\varepsilon}}) &\geq & \limsup _{M\rightarrow\infty}\mu_N^M\big(S(t)\big((F+B_{\varepsilon})\cap E^M_N\big)\big)
\\
& = &
\limsup _{M\rightarrow\infty}\mu_N^M\big((F+B_{\varepsilon})\cap E^M_N\big)
\\
& \geq &
\liminf _{M\rightarrow\infty}\mu_N^\perp(F+B_{\varepsilon})\geq \mu_N^\perp(F)\,.
\end{eqnarray*}
Letting $\varepsilon\rightarrow 0$ and using the Lebesgue theorem we get $\mu_N^\perp(F)\leq \mu_N^\perp(S(t)(F))$.
By the reversibility of $S(t)$, we get $\mu_N^\perp(F)= \mu_N^\perp(S(t)(F))$ for every closed set $F$ of $E^N$.
Finally by standard approximation arguments, we obtain that
$\mu_N^\perp(A)= \mu_N^\perp(S(t)(A))$ for every Borel set $A$ of $E^N$.
This completes the proof of Lemma~\ref{esp2}.
\end{proof}
Let us now turn to the proof of Proposition~\ref{ch-var}.
By definition we have the identities
\begin{equation}\label{def}
\pi_{N}\Phi^N_t=\tilde{\Phi}^N_t\pi_{N},
\quad
\pi_{>N}\Phi^N_t=S(t)\pi_{>N}.
\end{equation}
We can write
$$\gamma_N^{-1}\int_{\Phi^N_t(A)} F_N (u) d\mu=
\int_{\Phi^N_t(A)} 
H(\pi_{N} u)du_1...du_N \times d\mu^\perp_N
$$
where 
$$
H(\pi_{N} u)=\prod_{j=0}^{k-2} \chi_R (E_{j/2}(\pi_N u)) 
\chi_R (E_{(k-1)/2}(\pi_N u)-\alpha_N)  e^{-E_{k/2}(\pi_N(u))}.
$$
If we set $dL_N=du_1...du_N$ then we have
\begin{multline*}
\int_{\Phi^N_t(A)} 
H(\pi_{N} u)dL_N \times d\mu^\perp_N
=
\\
\int_{E_N}\int_{E_N^\perp}\11(\Phi^N_t(A))(\pi_{N}(u),\pi_{>N}(u))H(\pi_N u)dL_N \times d\mu^\perp_N
\end{multline*}
where again $\11$ denotes the indicator function of a measurable set. 
Using the Fubini theorem, we get
 \begin{multline*}
\int_{\Phi^N_t(A)} 
H(\pi_{N} u)dL_N \times d\mu^\perp_N
=
\\
\int_{E_N}H(\pi_N u)
\Big(
\int_{E_N^\perp}\11(\Phi^N_t(A))(\pi_{N}(u),\pi_{>N}(u)) d\mu^\perp_N\Big)dL_N.
\end{multline*}
By Lemma~\ref{esp2},
$$
\cdots=
\int_{E_N}H(\pi_N u)
\Big(
\int_{E_N^\perp}\11(\Phi^N_t(A))(\pi_{N}(u),S(t)\pi_{>N}(u)) d\mu^\perp_N\Big)dL_N.
$$
By another use of the Fubini theorem, we get
$$
\cdots=
\int_{E_N^\perp}
\Big(\int_{E_N}H(\pi_N u)\11(\Phi^N_t(A))(\pi_{N}(u),S(t)\pi_{>N}(u))dL_N \Big)d\mu^\perp_N
.$$
Now, Lemma~\ref{esp1} yields
$$
\cdots=
\int_{E_N^\perp}
\Big(\int_{E_N}H(
\tilde{\Phi}^N_t(\pi_N u))\11(\Phi^N_t(A))(
\tilde{\Phi}^N_t(\pi_{N}(u)),S(t)\pi_{>N}(u))dL_N \Big)d\mu^\perp_N.
$$
Coming back to \eqref{def}, we arrive at the identity
\begin{equation*}
\int_{\Phi^N_t(A)} 
H(\pi_{N} u)dL_N \times d\mu^\perp_N
=
\int_{H^\sigma}
H(\tilde{\Phi}^N_t(\pi_N u))\11(\Phi^N_t(A))(\Phi^N_t(u)) dL_N\times d\mu^\perp_N.
\end{equation*}
Since $\Phi^N_t$ is a bijection, we have that $\11(\Phi^N_t(A))(\Phi^N_t(u))=\11(A)(u)$.
We therefore obtain that
\begin{equation*}
\int_{\Phi^N_t(A)} 
H(\pi_{N} u)dL_N \times d\mu^\perp_N=\int_{A}H(\tilde{\Phi}^N_t(\pi_N u)) dL_N\times d\mu^\perp_N.
\end{equation*}
A final use of \eqref{def} completes the proof of Proposition~\ref{ch-var}.
\end{proof}
The next proposition plays a key role in our analysis.
\begin{prop}\label{lem1}
Let $t_0\in \R$. We have the following:
$$\lim_{N\rightarrow \infty}
\sup_{\substack{t\in [0, t_0]\\A\in {\mathcal B}(H^\sigma)}} 
\Big |\frac d{dt} \int_{\Phi_t^N(A)} F_N (u) d\mu\Big |=0.$$ 
\end{prop}
\begin{proof}
{\bf First step: estimate for $t=0$}
\\
\\
We have to show 
\begin{equation}\label{first_step}
\lim_{N\rightarrow \infty} \sup_{A\in {\mathcal B}(H^\sigma)}
\Big |\frac d{dt} \Big (\int_{\Phi^N_t(A)} F_N (u) d\mu \Big)_{t=0}\Big |=0.
\end{equation}
As a consequence of Proposition~\ref{ch-var}, we deduce
$$\frac d{dt}\Big(\int_{\Phi_t^N(A)} F_N (u) d\mu\Big)_{t=0}=
$$
$$\int_{A} G_N(u) \prod_{j=0}^{k-2} \chi_R (E_{j/2}(\pi_N (u))) 
\chi_R (E_{(k-1)/2}(\pi_N (u))-\alpha_N) 
e^{-R_{k/2}(\pi_N u)} d\mu+$$
$$+ \int_{A} H_N(u) \prod_{j=0}^{k-2} \chi_R (E_{j/2}(\pi_N (u))) 
\chi_R' (E_{(k-1)/2}(\pi_N (u))-\alpha_N) 
e^{-R_{k/2}(\pi_N u)} d\mu+$$
$$\sum_{j_0}\int_{A} L_N^{j_0}(u)  \chi_R' (E_{j_0/2}(\pi_N u)) \prod_{\substack{j=0\\ j\neq j_0}}^{k-2} \chi_R (E_{j/2}(\pi_N u)) 
\chi_R (E_{(k-1)/2}(\pi_N u)-\alpha_N)  \times $$
$$e^{-R_{k/2}(\pi_N u)} d\mu$$
where $G_N(u), H_N(u), L_N^{j_0}(u)$ for $j_0=0,...,k-2$
are respectively defined in \eqref{eq:G}, \eqref{eq:H} and \eqref{eq:L}.
Thanks to Proposition~\ref{ens} (recall that we are assuming $k=2(m+1)$) and the H\"older inequality, we obtain 
\eqref{first_step}.
\\
{\bf Second step: estimate for $\bar t\in (0, t_0)$}
\\
\\
We have
$$\frac d{dt} \Big (\int_{\Phi^N_t(A)} F_N (u) d\mu\Big)_{t=\bar t}
=\lim_{h\rightarrow 0} h^{-1}\Big (\int_{\Phi^N_{\bar t+h}(A)} F_N (u) d\mu -
\int_{\Phi^N_{\bar t}(A)} F_N (u) d\mu\Big)$$
$$=\lim_{h\rightarrow 0} h^{-1}\Big (\int_{\Phi_h^N\circ \Phi^N_{\bar t}(A)} F_N (u) d\mu -
\int_{\Phi^N_{\bar t}(A)} F_N (u) d\mu\Big)$$
and hence
$$\frac d{dt} \Big (\int_{\Phi^N_t(A)} F_N (u) d\mu
\Big )_{t=\bar t}= \frac d{dt}\Big (\int_{\Phi_t^N(\tilde A)} F_N (u) d\mu\Big )_{t=0}$$
where $\tilde A=\Phi_{\bar t}^N(A)$.
The result follows by the first step.
This completes the proof of Proposition~\ref{lem1}.
\end{proof}
\begin{lem}\label{lem2}
For any given $t_0\in \R$, $A\in {\mathcal B}(H^\sigma)$ we have:
$$\lim_{N\rightarrow \infty} \Big (\int_A F_N(u) d\mu - \int_{\Phi_{t}^N(A)} F_N(u) d\mu\Big )
=0,  \hbox{ } \forall 
t\in [0, t_0].$$ 
\end{lem}
\begin{proof}
It follows by the fundamental theorem of calculus in conjunction
with Proposition~\ref{lem1}.
\end{proof}
\begin{lem}\label{lem3}
For every $R>0$ there exists $\bar t=\bar t(R)>0$ such that for every compact set $K\subset H^\sigma$, 
with $K\subset B^\sigma(R)$ we have
 $$
 \int_K F(u)  d\mu\leq \int_{\Phi_{t}(K)} F(u)d\mu, \hbox{ } \forall 
t\in (-\bar t, \bar t).$$
\end{lem}
\begin{proof}
By Lemma \ref{lem2} we get 
$$\int_{\Phi_{t}^N(K)} F_N(u) d\mu =
\int_{K} F_N(u) d\mu + o(1), \hbox{ } \forall 
t\in \R
$$
where $\lim_{N\rightarrow \infty} o(1)=0$. Moreover
$F_N\rightarrow F$ in $L^1(d\mu)$ and we get
\begin{equation}\label{krai1}
\lim_{N\rightarrow \infty}
\int_{\Phi_{t}^N(K)} F_N(u) d\mu=\lim_{N\rightarrow \infty}
\int_{K} F_N(u) d\mu= \int_K F(u) d\mu, \hbox{ } \forall t\in \R.
\end{equation}
By Proposition~\ref{prop:det}
we get $\bar t=\bar t(R)>0$ such that for every $\epsilon>0$  there exists a suitable $N_0(\epsilon)$ with the property
\begin{equation}\label{eq:inequ}\sup_{N>N_0(\epsilon)} \int_{\Phi_{t}^N(K)} F(u)d\mu\leq 
\int_{\Phi_{t}(K)+ B^s(\epsilon)}
F(u) d\mu, \hbox{ } \forall 
t\in (- \bar t , \bar t).\end{equation}
We estimate the l.h.s. as follows:
\begin{equation}\label{eq:limsup}\sup_{N>N_0(\epsilon)} \int_{\Phi_{t}^N(K)} F(u)d\mu\geq \lim_{N\rightarrow \infty}
\int_{\Phi_{t}^N(K)} F(u) d\mu.
\end{equation}
On the other hands we have that $K$ is closed in $H^s$ and since $\Phi_t$ is a diffeomorphism on $H^s$
also $\Phi_t(K)$ is also closed in $H^s$. As a consequence we deduce
$$\bigcap_{\epsilon>0}  (\Phi_t(K) + B^s(\epsilon))= \Phi_t(K)$$
and hence by the Lebesgue theorem we deduce that 
the r.h.s. in \eqref{eq:inequ} converges to
$\int_{\Phi_{t}(K)}
F(u) d\mu$ as $\epsilon\rightarrow 0$. By combining this fact with
\eqref{eq:limsup}
then we get
$$\lim_{N\rightarrow \infty} \int_{\Phi^N_{t}(K)} F_N(u) d\mu\leq \int_{\Phi_{t}(K)} F(u)d\mu,
\hbox{ } \forall 
t\in (-\bar t, \bar t).$$
The proof of Lemma~\ref{lem3} can be completed by combining the last inequality with \eqref{krai1}.
\end{proof}
Next we iterate the last lemma to get the following statement.
\begin{lem}\label{lem3pak}
Let $t_0\in\R$. Then for every compact $K\subset H^\sigma$ 
we get
 $$
 \int_K F(u)  d\mu\leq \int_{\Phi_{t_0}(K)} F(u)d\mu.
$$
\end{lem}
\begin{proof}
We give the proof only for $t_0$ positive, the analysis for negative $t_0$ is completely analogous.
Notice that by Lemma \ref{lem:boundedness} we can fix $R>0$ such that
\begin{equation}\label{eq:minult147}
\{\Phi_t (K)| t\in [0, t_0]\}\subset B^\sigma(R).
\end{equation}
Next we consider $\bar t=\bar t(R)\in (0, t_0]$ given in Lemma \ref{lem3} and we 
choose $\tilde t$ such that $$\tilde t\in (0, \bar t] \hbox{ and }
\frac{t_0}{\tilde t}\in \N.$$
By Lemma~\ref{lem3} we get
$$
\int_K F(u)  d\mu\leq \int_{\Phi_{\tilde t}(K)} F(u)d\mu.$$
Notice that by \eqref{eq:minult147} we have that
$\Phi_{\tilde t}(K)\subset B^\sigma(R)$ hence Lemma~\ref{lem3} can be iterated
and we obtain
$$\int_{\Phi_{\tilde t}(K)} F(u) d\mu \leq \int_{\Phi_{\tilde t}(\Phi_{\tilde t}(K))} F(u) d\mu
= \int_{\Phi_{2 \tilde t}(K)} F(u) d\mu.$$
By repeating this argument $N_0$ times, where $N_0\tilde t=t_0$, we get
$$\int_{\Phi_{(j-1)\tilde t}(K)} F(u) d\mu \leq \int_{\Phi_{j \tilde t}(K)} F(u) d\mu,
\hbox{ } \forall j=1,..., N$$
and hence by the above chain of inequalities we deduce
$$\int_{K} F(u) d\mu \leq \int_{\Phi_{t_0}(K)} F(u) d\mu.$$
This completes the proof of Lemma~\ref{lem3pak}.
\end{proof}
Using the reversibility of the flow, we now obtain the statement.
\begin{lem}\label{lem3pakpak}
Let $t_0\in\R$. Then for every compact $K\subset H^\sigma$ we have
 $$
 \int_K F(u)  d\mu= \int_{\Phi_{t_0}(K)} F(u)d\mu.
$$
\end{lem}
\begin{proof}
Using  Lemma~\ref{lem3pak}, we can write
 $$
\int_{\tilde K} F(u) d\mu \leq \int_{\Phi_{-t_0}(\tilde K)} F(u) d\mu
$$
for every compact $\tilde K\subset H^\sigma$.
By choosing now $\tilde K= \Phi_{t_0}(K)$ (notice that it is compact since $K$ is compact and the flow 
$\Phi_{t_0}$ is a diffeomorphism), then we get
$$\int_{\Phi_{t_0}(K)} F(u) d\mu\leq \int_{K} F(u) d\mu.
$$
This completes the proof of Lemma~\ref{lem3pakpak}, since the opposite inequality is proved in
Lemma \ref{lem3pak}.
\end{proof}
Let us now complete the  proof of Theorem~\ref{lem4}.
Let $A$ be an arbitrary Borel set in $H^\sigma$.
It is well--known that there exists a sequence
of compact sets $K_n\subset A$ such that $$\lim_{n\rightarrow \infty} 
\int_{K_n} F(u) d\mu= \int_A F(u) d\mu.$$
On the other hands by Lemma \ref{lem3pakpak} we have
$$\int_{K_n} F(u) d\mu= \int_{\Phi_{t_0}(K_n)} F(u) d\mu\leq \int_{\Phi_{t_0}(A)} F(u) d\mu$$
(where at the last step we used the property $\Phi_{t_0}(K_n)\subset \Phi_{t_0}(A)$
in conjunction with the positivity  of $F(u)$).
As a consequence we get 
$$\int_A F(u) d\mu\leq \int_{\Phi_{t_0}(A)} F(u) d\mu.$$
The opposite inequality can be proved by using the reversibility of the flow in the same spirit as in
Lemma \ref{lem3pakpak}. This completes the proof of Theorem~\ref{lem4}.

\end{document}